\documentclass[12pt,leqno]{amsart}
\usepackage{amsmath}
\usepackage{amssymb}
\def\overset#1#2{{\mathrel{\mathop {{#2}_{}}\limits^{#1}}}}
\def\underset#1#2{{\mathrel{\mathop {{}_{} {#2}}\limits_{{#1}_{}}}}}
\def\upplim_#1{\underset{#1}{\overline\lim}\;}
\def\lowlim_#1{\underset{#1}{\underline\lim}\;}
\newtheorem{corollary}[equation]{Corollary}
\newtheorem{definition}[equation]{{\sc Definition}\rm }
\newtheorem{claim}[equation]{\indent{\it Claim}\rm }

\newtheorem{lemma}[equation]{Lemma}
\newtheorem{proposition}[equation]{Proposition}
\newtheorem{remark}[equation]{Remark}
\newtheorem{theorem}[equation]{Theorem}

\newcommand{\C}{{\mathbb{C}}}

\newcommand{\N}{\mathbb{N}}

\renewcommand{\P}{{\mathbb{P}}}
\newcommand{\B}{{\mathbb{B}}}

\newcommand{\R}{{\mathbb{R}}}

\newcommand{\Z}{\mathbb{Z}}

\numberwithin{equation}{section}

\title[Meromorphic mappings into projective varieties]{Meromorphic mappings into projective varieties with arbitrary families of moving hypersurfaces} 

\author{Si Duc Quang}

\begin{document}

\begin{abstract}
In this paper, we prove a general second main theorem for meromorphic mappings into a subvariety $V$ of $\P^N(\C)$ with an arbitrary family of moving hypersurfaces. Our second main theorem generalizes and improves all previous results for meromorphic mappings with moving hypersurfaces, in particular for meromorphic  mappings and families of moving hypersurfaces in subgeneral position. The method of our proof is different from that of previous authors used for the case of moving hypersurfaces. 
\end{abstract}

\def\thefootnote{\empty}

\maketitle

\section{Introduction}

In 1926, R. Nevanlinna \cite{N} initially established a second main theorem for meromorphic functions and fixed values in the complex plane $\C$. Later on, in 1930 H. Cartan \cite{Ca} extended the result of Nevanlinna to the case of linearly nondegenerate meromorphic mappings from $\C^m$ into $\P^N(\C)$ with hyperplanes in general position. By introducing the notion of Nochka's weights, E. Nochka \cite{Noc83} established the second main theorem for such mappings with families of hyperplanes in subgeneral position. 

In some few recent years, the second main theorem for hypersurfaces (fixed or moving) is studied intensively by many authors. Firstly, In 1992, A. E. Eremenko and M. L. Sodin \cite{ES} proved a second main theorem for arbitrary holomorphic curve from $\C$ into $\P^N(\C)$ intersecting $q$ hypersurfaces $\{Q_i\}_{i=1}^q$ in general position. In 2004 M. Ru \cite{Ru04} considered the case of algebraically nondegenerate meromorphic mappings from $\C^m$ into $\P^N(\C)$ with families of hypersurfaces in general position by using the filtration of the vector space of homogeneous polynomial of P. Corvaja and U. Zannier \cite{CZ}. Then, in \cite{Ru09} he generalized his results to the case of algebraically nondegenerate meromorphic mapping into a projective subvariety of $\P^N(\C)$ by using the method in Diophantine approximation proposed by J. Evertse and R. Ferretti \cite{EF1}. His results are generalized by S. D. Quang for the case of families of hypersurfaces in subgeneral position in \cite{Q19}, where he proposed the replacing hypersurfaces method to avoid using Nochka's weights. 

On the other hand, the results of M. Ru also generalized to the case of moving hypersurfaces by G. Dethloff and T. V. Tan \cite{DT1,DT2}. In oder to establish the second main theorem for the case of moving hypersurfaces in general position, G. Dethloff and T. V. Tan have constructed a new filtration instead of the filtration of Corvaja and Zannier. Also in 2018, S. D. Quang \cite{Q18} established the second main theorem for the case of moving hypersurfaces in subgeneral position. Recently Q. Yan and G. Yu \cite{YY} used the methods of \cite{Q18,Q19,DT2} to give a second main theorem for nonconstant meromorphic mappings, but their second main theorem does not give the normal form of Cartan's theorem as usual. Here, we will state here the normal form of the second main theorems given by Dethloff-Tan and Quang as follows. 

We recall the following. Denote by $\mathcal M$ the field of all meromorphic functions on $\C^m$. We call a moving hypersurface of degree $d$ in $\P^N(\C)$ a homogeneous polynomial $Q\in\mathcal M[x_0,\ldots,x_N]$ of the form:
$$ Q(z)({\bf x})=\sum_{I\in\mathcal T_{d}}a_{I}(z){\bf x}^I,$$
where $\mathcal T_d$ is the set of all $n$-tuples $(i_0,\ldots ,i_N)$ with $i_0+\cdots+ i_N=d$ and $i_j\ge 0$ for all $j$, $a_{I}\in\mathcal M\ (I\in\mathcal T_{d})$ are not all zeros, ${\bf x}=(x_0,\ldots,x_N)$, ${\bf x}^I=x_0^{i_0}\cdots x_N^{i_N}$ for all $I\in\mathcal T_d$. 
For each $z$, which is not pole of any $a_{I}$ and is not a common zero of all $a_{I}\ (I\in\mathcal T_{d})$, we denote by $Q(z)^*$ the support of $Q(z)$, i.e.,
$$ Q(z)^*=\{(x_0:\cdots:x_N)\in\P^N(\C)|\sum_{I\in\mathcal T_{d}}a_{I}(z){\bf x}^I=0\}.$$
We may consider $Q$ as a meromorphic mapping into $\P^{\binom{N+d}{N}-1}(\C)$ with a reduced representation $(\cdots :ha_{I}:\cdots )$ and denote by $T_{Q}(r)$ its characteristic function, where $h$ is a suitable meromorphic function.
The moving hypersurface $Q$ is said to be slow with respect to a meromorphic mapping $f$ from $\C^m$ into $\P^N(\C)$ if $\| T_{a_{I}}(r)=o(T_f(r))\ \forall I\in\mathcal T_d$. This implies that $\| T_{Q}(r)=o(T_f(r))$. Here the notion ``$\| P$'' means the assertion $P$ holds for all $r\in (0,+\infty)$ outside a finite Borel measure set.

Let $\{Q_i\}_{i=1}^q$ be a family of moving hypersurfaces in $\P^N(\C)$, $\deg Q_i=d_i$. Assume that
$$ Q_i(z)({\bf x})=\sum_{I\in\mathcal T_{d_i}}a_{iI}(z){\bf x}^I.$$

Denote by ${\mathcal K}_{\{Q_i\}_{i=1}^q}$ the smallest subfield of $\mathcal M$ (the field of all meromorphic functions on $\C^m$) which contains $\C$ and all $\frac{a_{iI}}{a_{iJ}}$ with $a_{iJ}\not\equiv 0$.

Let $V$ be a subvariety of $\P^N(\C)$ of dimension $n$. We say that $\{Q_i\}_{i=1}^q$ is in weakly general position in $V$ if there exists $z \in \C^m$ such that all $a_{iI}\ (1\le i\le q,\ I\in\mathcal I)$ are holomorphic at $z$, and 
$$ V\cap\bigcap_{j=0}^nQ_{i_j}(z)^*=\varnothing.$$

\vskip0.2cm
\noindent
\textbf{Theorem A} (see \cite{DT2}).
{\it Let $f$ be a holomorphic curve from $\C$ into a projective subvariety $V$ in $\P^N(\mathbf{C})$ of dimension $n$ with a reduced representation ${\bf f}=(f_0,\ldots,f_N)$. Let $\{Q_i\}_{i=1}^q$ be a set of slowly (with respect to $f$) moving hypersurfaces in weakly general position with respect to $V$ and $\deg Q_i = d_i\ (1\le i\le q).$ Assume that $f$ is algebraically nondegenerate over $\mathcal K_{\{Q_i\}_{i=1}^q}$.  Then for any $\epsilon >0$,
$$  \|\ (q-n-1-\epsilon)T_f(r)\le \sum_{i=1}^{q}\dfrac{1}{d_i}N_{Q_i({\bf f})}(r).$$}

In 2018, by using replacing hypersurface method, S. D. Quang \cite{Q18} considered the case of moving hypersurfaces in subgeneral position and got the following second main theorem.

\vskip0.2cm
\noindent
\textbf{Theorem B} (see \cite{Q18}).
{\it Let $f$ be a nonconstant meromorphic map of $\mathbf{C}^m$ into $\P^n(\mathbf{C})$ with a reduced representation ${\bf f}=(f_0,\ldots,f_n)$. Let $\{Q_i\}_{i=1}^q$ be a family of slowly (with respect to $f$) moving hypersurfaces in weakly $l-$subgeneral position with $\deg Q_i = d_i\ (1\le i\le q).$ Assume that $f$ is algebraically nondegenerate over $\mathcal K_{\{Q_i\}_{i=1}^q}$.  Then for any $\epsilon >0$, 
$$  \|\ (q-(l-n+1)(n+1)-\epsilon)T_f(r)\le \sum_{i=1}^{q}\dfrac{1}{d_i}N^{[L_j]}_{Q_i({\bf f})}(r)+o(T_f(r)),$$
where $L_j=\frac{1}{d_j}L_0$ and $L_0$ is explicitly estimated.}

Our purpose in this paper is to generalize all these above mentioned results to the case of moving hypersurfaces into a projective subvariety $V$ and an arbitray family of hypersurfaces $\{Q_i\}_{i=1}^q$. In order to do so, in Section 3, we introduced the notion of ``distributive constant'' $\Delta_V$ (resp. $\Delta_f$) of a family of moving hypersurfaces with respect to $V$ (resp. with respect to a meromorphic mapping $f$), and the notion of ``algebraic dimension $n_f$'' and ``algebraic degree $\delta_f$'' over the field $\mathcal K_{\{Q_i\}_{i=1}^q}$ of the mapping $f$.
Our main result is stated as follows.

\begin{theorem}\label{1.1}  
Let $f$ be a nonconstant meromorphic map of $\mathbf{C}^m$ into $\P^N(\mathbf{C})$ with a reduced representation ${\bf f}=(f_0,\ldots,f_N)$. Let $\{Q_i\}_{i=1}^q$ be a family of slowly (with respect to $f$) moving hypersurfaces in $\P^N(\C)$ with $\deg Q_i = d_i\ (1\le i\le q).$ Let $n_f$ be the algebraic dimension of $f$ over $\mathcal K_{\{Q_i\}_{i=1}^q}$ and let $\Delta_f$ be the distributive constant of $\{Q_i\}_{i=1}^q$ with respect to $f$. Suppose that $f$ has finite algebraic degree over $\mathcal K_{\{Q_i\}_{i=1}^q}$.  Then for any $\epsilon >0$, 
$$  \|\ (q-\Delta_f(n_f+1)-\epsilon)T_f(r)\le \sum_{i=1}^{q}\dfrac{1}{d_i}N_{Q_i({\bf f})}(r).$$
\end{theorem}
Hence, if $\Delta_f=1$, i.e., $\{Q_i\}_{i=1}^q$ is in weakly general position with respect to $f$, we will get the sharp second main theorem
$$  \|\ (q-(n_f+1)-\epsilon)T_f(r)\le \sum_{i=1}^{q}\dfrac{1}{d_i}N_{Q_i({\bf f})}(r).$$
If we assume further that $f$ comes into a subvariety $V$ of $\P^N(\C)$ of dimension $n$, and $f$ is algebraically nondegenerate over $\mathcal K$, then $n_f=\dim V$, $\Delta_f=\Delta_V$, the algebraic degree of $f$ over $\mathcal K$ is equal to the degree of $V$. Hence, we have the following corrolary.
\begin{corollary}\label{1.2}
Let $f$ be a nonconstant meromorphic map of $\mathbf{C}^m$ into a subvariety $V$ of $\P^N(\mathbf{C})$ with a reduced representation ${\bf f}=(f_0,\ldots,f_N)$. Let $\{Q_i\}_{i=1}^q$ be a family of slowly (with respect to $f$) moving hypersurfaces with $\deg Q_j = d_j\ (1\le i\le q).$ Assume that $f$ is algebraically nondegenerate over $\mathcal K_{\{Q_i\}_{i=1}^q}$.  Let $\Delta_V$ be the distributive constant of $\{Q_i\}_{i=1}^q$ with respect to $V$. Then, for every $\epsilon>0$,
$$  \|\ (q-\Delta_V(\dim V+1)-\epsilon)T_f(r)\le \sum_{i=1}^{q}\dfrac{1}{d_i}N_{Q_i({\bf f})}(r).$$
\end{corollary}
We see that, if $\Delta_V=1$ (i.e., $\{Q_i\}_{i=1}^q$ is in weakly general position with respect to $V$) we will get Theorem A. If $\{Q_i\}_{i=1}^q$ is in weakly $l$-subgeneral position with respect to $V$ then $\Delta_V\le (l-n+1)$, and hence we get an extension of Theorem B.

We would also like to emphasize that, our proof is different from the proofs of the previous authors, such as Dethloff-Tan \cite{DT1,DT2} and Yan-Yu \cite{YY}. In fact, we will develop the method making use of Chow weight of Evertse-Ferretti \cite{EF1} and also Ru \cite{Ru04,Ru09}. Our proof is simpler than that in \cite{DT1,DT2,YY}. Also, from our proof, it is impossible to give the truncation level for the counting functions.
     
\section{Notation and Auxialiary results}

\noindent
\textbf{(a) Some notation from Nevanlina theory.}

We use the following usual notations:
\begin{align*}
 \|z\|& := \big(|z_1|^2 + \dots + |z_m|^2\big)^{1/2} \text{ for }z = (z_1,\dots,z_m) \in \mathbb C^m,\\
\B^m(r) &:= \{ z \in \mathbb C^m : \|z\| < r\},\\
S(r) &:= \{ z \in \mathbb C^m : \|z\| = r\}\ (0<r<\infty),\\
v_{m-1}(z) &:= \big(dd^c \|z\|^2\big)^{m-1},\\
\sigma_m(z)&:= d^c\log\|z\|^2 \land \big(dd^c\log\|z\|^2\big)^{m-1} \text{on} \quad \mathbb C^m \setminus \{0\}.
\end{align*}
 For a divisor $\nu$ on  a $\mathbb C^m$ and a positive integer $M$ or $M= +\infty$, as usual we denote by $N^{[M]}(r,\nu)$ the counting function of $\nu$ with multiplicities truncated to level $M$.

For a meromorphic function $\varphi$ on $\C^m$, denote by $\nu_\varphi$ its divisor of zeros and set
$$N_{\varphi}(r)=N(r,\nu_{\varphi}), \ N_{\varphi}^{[M]}(r)=N^{[M]}(r,\nu_{\varphi})\ (r_0<r<R_0).$$

For brevity, we will omit the character $^{[M]}$ if $M=+\infty$.

Let $f : \mathbf{C}^m \longrightarrow \P^N(\mathbf{C})$ be a meromorphic mapping with a reduced representation
${\bf f}= (f_0,  \ldots , f_N)$. Set $\|{\bf f}\| = \big(|f_0|^2 + \dots + |f_N|^2\big)^{1/2}$.
The characteristic function of $f$ is defined by 
\begin{align*}
T_f(r)= \int\limits_{S(r)} \log\|{\bf f}\| \sigma_m -
\int\limits_{S(1)}\log\|{\bf f}\|\sigma_m.
\end{align*}

Let $\varphi$ be a nonzero meromorphic function on $\mathbf{C}^m$, which are occasionally regarded as a meromorphic map into $\P^1(\mathbf{C})$. The proximity function of $\varphi$ is defined by
$$m(r,\varphi):=\int_{S(r)}\log \max\ (|\varphi|,1)\sigma_m.$$
The Nevanlinna's characteristic function of $\varphi$ is defined as follows
$$ T(r,\varphi):=N_{\frac{1}{\varphi}}(r)+m(r,\varphi). $$
Then 
$$T_\varphi (r)=T(r,\varphi)+O(1).$$
The function $\varphi$ is said to be small (with respect to $f$) if $\|\ T_\varphi (r)=o(T_f(r))$.

Let $Q$ be a moving hypersurface in $\P^N(\mathbf{C})$ of degree $d\ge 1$ given by
$$ Q(z)({\bf x})=\sum_{I\in\mathcal T_d}a_I{\bf x}^I, $$
where $a_I\in\mathcal M$ for all $I\in\mathcal T_d$. 
The proximity function of $f$ with respect to $Q$, denoted by $m_f (r,Q)$, is defined by
$$m_f (r,Q)=\int_{S(r)}\log\dfrac{\|{\bf f}\| ^d}{|Q({\bf f})|}\sigma_m-\int_{S(r_0)}\log\dfrac{\|{\bf f}\| ^d}{|Q({\bf f})|}\sigma_m.$$
This definition is independent of the choice of the reduced representation of $f$. 

If $Q$ is a slowly moving hypersurface with respect to $f$, then  the first main theorem in Nevanlinna theory for meromorphic mappings and moving hypersurfaces is stated as follows:
$$dT_f (r)=m_f (r,Q) + N_{Q({\bf f})}(r)+o(T_f(r)).$$

\noindent

\begin{lemma}[{Lemma on logarithmic derivative, see \cite{NO}}]\label{2.1}
Let $f$ be a nonzero meromorphic function on $\mathbf{C}^m.$ Then 
$$\biggl|\biggl|\quad m\biggl(r,\dfrac{\mathcal{D}^\alpha (f)}{f}\biggl)=O(\log^+T(r,f))\ (\alpha\in \Z^m_+).$$
\end{lemma}

Repeating the argument in (Prop. 4.5 \cite{Fu}), we have the following.

\begin{proposition}[{see \cite[Prop. 4.5]{Fu}}]\label{2.2}
Let $\Phi_1,...,\Phi_k$ be meromorphic functions on $\mathbf{C}^m$ such that $\{\Phi_1,...,\Phi_k\}$ 
are  linearly independent over $\mathbf{C}.$ Then  there exists an admissible set, which is uniquely chosen in an explicitly way,  
$$\{\alpha_i=(\alpha_{i1},...,\alpha_{im})\}_{i=1}^k \subset \Z^m_+$$
with $|\alpha_i|=\sum_{j=1}^{m}|\alpha_{ij}|\le i-1 \ (1\le i \le k)$ such that the following are satisfied:

(i)\  $\{{\mathcal D}^{\alpha_i}\Phi_1,...,{\mathcal D}^{\alpha_i}\Phi_k\}_{i=1}^{k}$ is linearly independent over $\mathcal M,$\ i.e., \ $\det{({\mathcal D}^{\alpha_i}\Phi_j)}\not\equiv 0,$ 

(ii) $\det \bigl({\mathcal D}^{\alpha_i}(h\Phi_j)\bigl)=h^{k}\cdot \det \bigl({\mathcal D}^{\alpha_i}\Phi_j\bigl)$ for
any nonzero meromorphic function $h$ on $\mathbf{C}^m.$
\end{proposition}
The determinant $\det \bigl({\mathcal D}^{\alpha_i}\Phi_j\bigl)$ is usually called the general Wronskian of $\{\Phi_1,...,\Phi_k\}$.

\vskip0.2cm
\noindent
{\bf (b) Chow weights and Hilbert weights.} 

We recall the notion of Chow weights and Hilbert weights from \cite{Ru09}.

Let $X\subset\P^N(\C)$ be a projective variety of dimension $n$ and degree $\delta$. The Chow form of $X$ is the unique polynomial, up to a constant scalar, 
$$F_X(\textbf{u}_0,\ldots,\textbf{u}_n) = F_X(u_{00},\ldots,u_{0N};\ldots; u_{n0},\ldots,u_{nN})$$
in $N+1$ blocks of variables $\textbf{u}_i=(u_{i0},\ldots,u_{iN}), i = 0,\ldots,n$ with the following
properties: 
\begin{itemize}
\item $F_X$ is irreducible in $k[u_{00},\ldots,u_{nN}]$;
\item $F_X$ is homogeneous of degree $\delta$ in each block $\textbf{u}_i, i=0,\ldots,n$;
\item $F_X(\textbf{u}_0,\ldots,\textbf{u}_n) = 0$ if and only if $X\cap H_{\textbf{u}_0}\cap\cdots\cap H_{\textbf{u}_n}\ne\varnothing$, where $H_{\textbf{u}_i}, i = 0,\ldots,n$, are the hyperplanes given by
$$u_{i0}x_0+\cdots+ u_{iN}x_N=0.$$
\end{itemize}
Let ${\bf c}=(c_0,\ldots, c_N)$ be a tuple of real numbers and $t$ be an auxiliary variable. We consider the decomposition
\begin{align*}
F_X(t^{c_0}u_{00},&\ldots,t^{c_N}u_{0N};\ldots ; t^{c_0}u_{n0},\ldots,t^{c_N}u_{nN})\\ 
& = t^{e_0}G_0(\textbf{u}_0,\ldots,\textbf{u}_N)+\cdots +t^{e_r}G_r(\textbf{u}_0,\ldots, \textbf{u}_N),
\end{align*}
with $G_0,\ldots,G_r\in \C[u_{00},\ldots,u_{0N};\ldots; u_{n0},\ldots,u_{nN}]$ and $e_0>e_1>\cdots>e_r$. The Chow weight of $X$ with respect to ${\bf c}$ is defined by
\begin{align*}
e_X({\bf c}):=e_0.
\end{align*}
For each subset $J = \{j_0,\ldots,j_n\}$ of $\{0,\ldots,N\}$ with $j_0<j_1<\cdots<j_n,$ we define the bracket
\begin{align*}
[J] = [J]({\bf u}_0,\ldots,{\bf u}_n):= \det (u_{ij_t}), i,t=0,\ldots,n,
\end{align*}
where $\textbf{u}_i = (u_{i0},\ldots,u_{iN})\ (1\le i\le n)$ denote the blocks of $N+1$ variables. Let $J_1,\ldots,J_\beta$ with $\beta=\binom{N+1}{n+1}$ be all subsets of $\{0,\ldots,N\}$ of cardinality $n+1$.

Then $F_X$ can be written as a homogeneous polynomial of degree $\delta$ in $[J_1],\ldots,[J_\beta]$. We may see that for $\textbf{c}=(c_0,\ldots,c_N)\in\R^{N+1}$ and for any $J$ among $J_1,\ldots,J_\beta$,
\begin{align*}
\begin{split}
[J](t^{c_0}u_{00},\ldots,t^{c_N}u_{0N},&\ldots,t^{c_0}u_{n0},\ldots,t^{c_N}u_{nN})\\
&=t\sum_{j\in J}c_j[J](u_{00},\ldots,u_{0N},\ldots,u_{n0},\ldots,u_{nN}).
\end{split}
\end{align*}

For $\textbf{a} = (a_0,\ldots,a_N)\in\mathbb Z^{N+1}$ we write ${\bf x}^{\bf a}$ for the monomial $x^{a_0}_0\cdots x^{a_N}_N$. Denote by $\C[x_0,\ldots,x_N]_u$ the vector space of homogeneous polynomials in $\C[x_0,\ldots,x_N]$ of degree $u$ (including $0$). For an ideal $I$ in $\C[x_0,\ldots,x_N]$, we put $I_u :=\C[x_0,\ldots,x_N]_u\cap I$. Let $I(X)$ be the prime ideal in $\C[x_0,\ldots,x_N]$ defining $X$.  The Hilbert function $H_X$ of $X$ is defined by, for $u = 1, 2,\ldots,$
\begin{align*}
H_X(u):=\dim (\C[x_0,\ldots,x_N]_u/I(X)_u).
\end{align*}
By the usual theory of Hilbert polynomials,
\begin{align*}
H_X(u)=\delta\cdot\frac{u^N}{N!}+O(u^{N-1}).
\end{align*}
The $u$-th Hilbert weight $S_X(u,{\bf c})$ of $X$ with respect to the tuple ${\bf c}=(c_0,\ldots,c_N)\in\mathbb R^{N+1}$ is defined by
\begin{align*}
S_X(u,{\bf c}):=\max\left (\sum_{i=1}^{H_X(u)}{\bf a}_i\cdot{\bf c}\right),
\end{align*}
where the maximum is taken over all sets of monomials ${\bf x}^{{\bf a}_1},\ldots,{\bf x}^{{\bf a}_{H_X(u)}}$ whose residue classes modulo $I$ form a basis of $k[x_0,\ldots,x_N]_u/I_u.$

The following theorems are due to J. Evertse and R. Ferretti.
\begin{theorem}[{see  \cite[Theorem 4.1]{EF1}}]\label{2.3}
Let $X\subset\P^N(\C)$ be an algebraic variety of dimension $n$ and degree $\delta$. Let $u>\delta$ be an integer and let ${\bf c}=(c_0,\ldots,c_N)\in\mathbb R^{N+1}_{\geqslant 0}$.
Then
$$ \frac{1}{uH_X(u)}S_X(u,{\bf c})\ge\frac{1}{(n+1)\delta}e_X({\bf c})-\frac{(2n+1)\delta}{u}\cdot\left (\max_{i=0,\ldots,N}c_i\right). $$
\end{theorem}

\begin{lemma}[{see \cite[Lemma 5.1]{EF2}, also \cite{Ru09}}]\label{2.4}
Let $Y\subset\P^N(\C)$ be an algebraic variety of dimension $n$ and degree $\delta$. Let ${\bf c}=(c_1,\ldots, c_q)$ be a tuple of positive reals. Let $\{i_0,\ldots,i_n\}$ be a subset of $\{1,\ldots,q\}$ such that
$$Y \cap \{y_{i_0}=\cdots =y_{i_n}=0\}=\varnothing.$$
Then
$$e_Y({\bf c})\ge (c_{i_0}+\cdots +c_{i_n})\delta.$$
\end{lemma}

\section{Distributive constant and lemma on replacing moving hypersurfaces}

Let $Q_1,\ldots,Q_q$ be $q$ moving hypersurfaces in $\P^N(\C)$ given by
$$ Q_i(z)({\bf x})=\sum_{I\in\mathcal T_{d_i}}a_{iI}(z){\bf x}^I,$$
where ${\bf x}=(x_0,\ldots,x_N)$, ${\bf x}^I=x_0^{i_0}\cdots x_N^{i_N}$ for $I=(i_0,\ldots,i_N)$. Denote simply by $\mathcal K$ the field $\mathcal K_{\{Q_i\}_{i=1}^q}$. 
Denote by $\mathcal C_{\mathcal K}$ the set of all non-negative functions $h : \mathbf{C}^m\setminus A\longrightarrow [0,+\infty]$, which are of the form
$$ h=\dfrac{|g_1|+\cdots +|g_l|}{|g_{l+1}|+\cdots +|g_{l+k}|}, $$
where $k,l\in\N,\ g_1,...., g_{l+k}\in\mathcal K\setminus\{0\}$ and $A\subset\mathbf{C}^m$, which may depend on
$g_1,....,g_{l+k}$, is an analytic subset of codimension at least two. Then, for $h\in\mathcal C_{\mathcal K}$ we have
$$\int\limits_{S(r)}\log h\sigma_m=O(\max T_{a_{iI}/a_{iJ}}(r)).$$
Then, we see that for every moving hypersurface $Q$ in $\mathcal K[x_0,\ldots,x_N]$ of degree $d$, we have
$$ Q(z)({\bf x})\le c(z)\|{\bf x}\|^d $$
for some $c\in\mathcal C_{\mathcal K}$.

Let $\{R_1,\ldots,R_p\}$ be a family of nonzero homogeneous polynomials in $\mathcal K[x_0,\ldots,x_N]$. For a point $z$, which is neither zero nor pole of any nonzero coefficients of $R_i\ (1\le i\le p)$, we set
$$ n_z=\dim\bigcap_{i=1}^pR_i(z)^*.$$
\begin{lemma}\label{new3.1}
$n_z$ are constant outside a proper analytic subset of $\C^m$.
\end{lemma}
\begin{proof} Take a point $z_0$ so that 
$$n:=n_{z_0}=\min_{z}\dim\bigcap_{i=1}^pR_i(z)^*,$$
where the minimum is taken over all $z$, which is neither zero nor pole of any nonzero coefficients of $R_i\ (1\le i\le q)$. Choose $H_1,\ldots,H_{n+1}$ be $n+1$ hyperplanes in $\P^N(\C)$ such that 
$$\bigcap_{i=1}^pR_i(z_0)^*\cap\bigcap_{i=1}^{n+1}H_i^*=\varnothing.$$
This implies that 
$$\bigcap_{i=1}^pR_i(z)^*\cap\bigcap_{i=1}^{n+1}H_i^*=\varnothing$$
for all $z$ outside a proper analytic subset of $\C^m$. But this assertion yields that
$$ \dim\bigcap_{i=1}^pR_i(z)^*\le n.$$
By the definition of $n$, we have
$$ \min_{z}\dim\bigcap_{i=1}^pR_i(z)^*=n$$
for all $z$ outside a proper analytic subset of $\C^m$. The lemma is proved.
\end{proof}

Let $f$ be a meromorphic mapping from $\C^m$ into $\P^N(\C)$ with a reduced representation ${\bf f}=(f_0,f_1,\ldots,f_N)$. Assume that $Q_i(z)({\bf f}(z))\not\equiv 0$ for all $1\le i\le q$. Denote by $I(f(\C^m))$ the ideal of all polynomials $Q$ in $\mathcal K[x_0,\ldots,x_N]$ (including zero polynomial) such that $Q({\bf f})\equiv 0$. By the coherent of the ring of $\mathcal K[x_0,\ldots,x_N]$, there exists a minimal generating set of $I(f(\C^m))$, denoted again by $\{R_1,\ldots,R_p\}$. Then there is a positive integer $n$ such that
$$ n=\dim\bigcap_{i=1}^pR_i(z)^*$$
for generic points $z\in\C^m$. Here, once we say an assertion holds for generic points, it means that assertion holds for every point in $\C^m$ outside a proper analytic subset.

Now, take another minimal generating set $R_1',\ldots,R_{p'}'$ of $I(f(\C^m))$. Let $n'$ be the integer such that
$$ n'=\dim\bigcap_{i=1}^{p'}R'_i(z)^*$$
for generic points $z$. We have
$$ R'_i=\sum_{j=1}^pa_{ij}R_i, $$
where $a_{ij}\in\mathcal K[x_0,\ldots,x_N]$. Therefore, for every point $z\in\C^m$, which is neither zero nor pole of any nonzero functions $a_{ij}$ and any nonzero coefficients of $R_j,R'_i$, we have
$$\bigcap_{i=1}^pR_i(z)^*\subset\bigcap_{i=1}^{p'}R'_i(z)^*.$$
Hence $n\le n'$. Similarly $n'\le n$, and hence $n=n'$.

By this reason, we may define the \textit{algebraic dimension} of $f$ over the field $\mathcal K$ by a constant $n_f$ such that
$$ n_f:=\dim\bigcap_{i=1}^pR_i(z)^*$$
for a minimal generating set $\{R_1,\ldots,R_p\}$ of $I(f(\C^m))$ and for generic points $z\in\C^m$.

Also, by the same arguments, we have
$$\bigcap_{i=1}^pR_i(z)^*=\bigcap_{i=1}^{p'}R'_i(z)^*$$
for generic points $z\in\C^m$. Then we may define the \textit{algebraic degree} of the mapping $f$ as follows:
$$ \delta_f=\min\left\{a\in\Z_+|a\ge\deg\bigcap_{i=1}^pR_i(z)^*\ \forall z\in\C^m\text{ outside a proper analytic subset}\right\}$$
(here, if there is no positive integer $a$ such that $a\ge\deg\bigcap_{i=1}^pR_i(z)^*$ for generic points $z$ then $\delta_f=+\infty$).
Hence if $f$ is considered as an algebraically nondegenerate (over $\mathcal K$) mapping into a subvariety $V$ of $\P^N(\C)$ then $\delta_f=\deg V$.

\begin{lemma}\label{new2} With the above notation, let $1\le j_1\le\cdots\le j_k\le q$. Suppose that there exists $z_0$, which is not pole of any coefficients of all $R_i\ (1\le i\le p),Q_{j_s}\ (1\le s\le k)$, such that $\bigcap_{i=1}^pR_i(z_0)^*\cap\bigcap_{s=1}^kQ_{j_s}(z_0)^*=\varnothing.$ Then we have
$$ \bigcap_{i=1}^pR_i(z)^*\cap\bigcap_{s=1}^kQ_{j_s}(z)^*=\varnothing$$
for generic points $z$, and there exists a function $c\in\mathcal C_{\mathcal K}$ such that
$$ \|{\bf f}(z)\|\le c(z)\max_{1\le s\le k}\{Q_{j_s}(z)({\bf f}(z))\}.$$
Moreover, 
$$ \bigcap_{i=1}^{p'}R'_i(z)^*\cap\bigcap_{s=1}^kQ_{j_s}(z)^*=\varnothing.$$
for generic points $z\in\C^m$.
\end{lemma}
\begin{proof} We use the same arguments as in \cite[Lemma 2.3]{DT2}. It is clear that
$$ \bigcap_{i=1}^pR_i(z)^*\cap\bigcap_{s=1}^kQ_{j_s}(z)^*=\varnothing$$
for generic points $z$. From \cite[Lemma 2.2]{DT1}, there exists a function $c\in\mathcal C_{\mathcal K}$ such that
\begin{align*}
\|{\bf f}(z)\|&\le c(z)\max\{R_1(z)({\bf f}(z)),\ldots,R_p(z)({\bf f}(z)),Q_{j_1}(z)({\bf f}(z)),\ldots,Q_{j_k}(z)({\bf f}(z))\}\\
&=c(z)\max_{1\le s\le k}\{Q_{j_s}(z)({\bf f}(z))\}.
\end{align*}
Also, as above arguments, we have 
$$\bigcap_{i=1}^{p'}R'_i(z)^*\cap\bigcap_{s=1}^kQ_{j_s}(z)^*\subset \bigcap_{i=1}^{p}R_i(z)^*\cap\bigcap_{s=1}^kQ_{j_s}(z)^*=\varnothing$$
for generic points $z\in\C^m$.
\end{proof}

For a subvariety $V$ and an analytic subset $S$ of $\P^N(\C)$, the codimension of $S$ in $V$ is defined by
$$ \mathrm{codim}_VS=\dim V-\dim (V\cap S).$$
\begin{definition}\label{3.3}
With the above notation, we define the distributive constant of the family $\{Q_1,\ldots,Q_q\}$ with respect to $f$ by
$$ \Delta_f:=\underset{\Gamma\subset\{1,\ldots,q\}}\max\dfrac{\sharp\Gamma}{\mathrm{codim}_{\bigcap_{i=1}^{p}R_i(z)^*}\left(\bigcap_{j\in\Gamma} Q_j(z)^*\right)}$$
for generic points $z\in\C^m$.
\end{definition}
Here, we note that $\dim\varnothing =-\infty$. By Lemma \ref{new3.1} and the above arguments, we see that the above definition is well-defined and does not depend on the choice of the minimal generating set $R_1,\ldots,R_p$ of $I(f(\C^m)).$

\begin{definition}
Let $V$ be a subvariety of $\P^N(\C)$ of dimension $n$. Assume that $V$ is not contained in any $Q_j(z)^*\ (1\le j\le q)$ generically. We define the distributive constant of $\{Q_1,\ldots,Q_q\}$ with respect to $V$ by
$$ \Delta_V:=\underset{\Gamma\subset\{1,\ldots,q\}}\max\dfrac{\sharp\Gamma}{\mathrm{codim}_V\left (\bigcap_{j\in\Gamma} Q_j(z)^*\right)} $$
for generic points $z\in\C^m$. 
\end{definition}

\begin{remark}\label{3.5}
{\rm If the image of $f$ is contained in $V$ and $Q_j\not\in I(f(\C^m))$ for all $j\in \{1,\ldots,q\}$ then $\Delta_f\le (\dim V-n_f+1)\Delta_V$.}
\end{remark}
Indeed, since $f(\C^m)\subset V$ and $Q_j\not\in I(f(\C^m))$ for all $j\in \{1,\ldots,q\}$, then $\bigcap_{i=1}^{p}R_i(z)^*\subset V$ and $\bigcap_{i=1}^{p}R_i(z)^*\not\subset\bigcap_{j\in\Gamma} Q_j(z)^*$ generically, and hence it is clear that
\begin{align*}
&\left(\dim V-\dim\bigcap_{i=1}^{p}R_i(z)^*+1\right)\left(\dim\bigcap_{i=1}^{p}R_i(z)^*-\dim\bigcap_{i=1}^{p}R_i(z)^*\cap\bigcap_{j\in\Gamma} Q_j(z)^*\right)\\
&\ge \dim V-\dim\bigcap_{i=1}^{p}R_i(z)^*+\dim\bigcap_{i=1}^{p}R_i(z)^*-\dim\bigcap_{i=1}^{p}R_i(z)^*\cap\bigcap_{j\in\Gamma} Q_j(z)^*\\
&\ge\dim V-\dim\bigcap_{j\in\Gamma} Q_j(z)^*,
\end{align*}
i.e.,
\begin{align*}
\left(\dim V-\dim\bigcap_{i=1}^{p}R_i(z)^*+1\right)\cdot\mathrm{codim}_{\bigcap_{i=1}^{p}R_i(z)^*}\left(\bigcap_{j\in\Gamma} Q_j(z)^*\right)\ge \mathrm{codim}_V\left (\bigcap_{j\in\Gamma} Q_j(z)^*\right)
\end{align*}
for generic points $z$ and all $\Gamma\subset\{1,\ldots,q\}.$ This straightforwardly implies that 
$$\Delta_f\le (\dim V-n_f+1)\Delta_V.$$

\begin{definition} With the above definition, the family $\{Q_1,\ldots,Q_q\}$ is said to be in weakly $l-$subgeneral position with respect to $V$ if for every $1\le j_0<\cdots<j_l\le q,$
$$ \bigcap_{s=0}^lQ_s(z)^*\cap V=\varnothing $$
for generic points $z\in\C^m$. The family $\{Q_1,\ldots,Q_q\}$ is said to be in weakly $l-$subgeneral position with respect to $f$ if for every $1\le j_0<\cdots<j_l\le q,$
$$ \bigcap_{s=0}^lQ_{j_s}(z)^*\cap \bigcap_{i=1}^pR_i(z)^*=\varnothing $$
for generic points $z\in\C^m$. If $l=\dim V$ (resp. $l=n_f$) then we say that $\{Q_1,\ldots,Q_q\}$ is in general position with respect to $V$ (resp. with respect to $f$).
\end{definition}
\begin{remark}\label{new38}{\rm

(a) If $Q_1,\ldots,Q_q\ (q\ge m+1)$ are in weakly $l-$subgeneral position with respect to $V$ then we may see that for every subset $\{Q_{j_1},\ldots,Q_{j_k}\}\ (1\le k\le l)$, one has
$$ \dim\bigcap_{i=1}^kQ_{j_i}(z)^*\le \min\{\dim V-1,l-k\},$$
generically, and hence
\begin{align*}
\Delta_V&\le\max\biggl\{\dfrac{1}{\dim V-(\dim V-1)},\ldots,\dfrac{l-\dim V+1}{\dim V-(\dim V-1)},\dfrac{l-\dim V+2}{\dim V-(\dim V-2)},\\
&\hspace{160pt}\ldots,\dfrac{l}{\dim V-(l-l)}\biggl\}=l-\dim V+1.
\end{align*}
(b)  If $Q_1,\ldots,Q_q\ (q\ge l+1)$ are in weakly $l-$subgeneral position with respect to $f$, then similarly we have 
$$ \Delta_f\le l-\dim\bigcap_{i=1}^pR_i(z)^*+1=l-n_f+1\ (\text{for generic points }z\in\C^m). $$}
\end{remark}

The following two lemmas play essential role in our proofs, which have been used in our recent work \cite{Q21} for fixed hypersufaces.
\begin{lemma}\label{3.2}
Let $V$ be a projective subvariety of $\P^N(\C)$ of dimension $n$. Let $Q_0,\ldots,Q_{l}$ be $l$ hypersurfaces in $\P^N(\C)$ of the same degree $d\ge 1$, such that $\bigcap_{i=0}^{l}Q_i^*\cap V=\varnothing$ and
$$\dim\left (\bigcap_{i=0}^{s}Q_i^*\right )\cap V=n-u\ \forall t_{u-1}\le s<t_u,1\le u\le n,$$
where $t_0,t_1,\ldots,t_n$ integers with $0=t_0<t_1<\cdots<t_n=l$. Then there exist $n+1$ hypersurfaces $P_0,\ldots,P_n$ in $\P^N(\C)$ of the forms
$$P_u=\sum_{j=0}^{t_{u}}c_{uj}Q_j, \ c_{uj}\in\C,\ u=0,\ldots,n,$$
such that $\left (\bigcap_{u=0}^{n}P_u^*\right )\cap V=\varnothing.$
\end{lemma}
\begin{proof} Set $P_0=Q_0$. We will construct $P_1,\ldots,P_n$ as follows.

Step 1. Firstly, we will construct $P_1$. For each irreducible component $\Gamma$ of dimension $n-1$ of $P_0^*\cap V$, we put 
$$V_{1\Gamma}=\{c=(c_0,\ldots,c_{t_1})\in \C^{t_1+1}\ ;\ \Gamma\subset Q_c^*,\text{ where }Q_c=\sum_{j=0}^{t_1}c_jQ_j\}.$$
Here, $Q_c$ may be zero polynomial and its support $Q^*_c$ is $\P^N(\C)$. We see that $V_{1\Gamma}$ is a subspace of $\C^{t_1+1}$. Since $\dim \left(\bigcap_{j=0}^{t_1}Q_j^*\right)\cap V\le n-2$, there exists $i \ (1\le i\le t_1+1)$ such that $\Gamma\not\subset Q_i^*$. Then $V_{1\Gamma}$ is a proper subspace of $\C^{t_1+1}$. Since the set of irreducible components of dimension $n-1$ of $P_0^*\cap V$ is at most countable, 
$$ \C^{t_1+1}\setminus\bigcup_{\Gamma}V_{1\Gamma}\ne\varnothing. $$
Hence, there exists $(c_{10},c_{11},\ldots,c_{1t_1})\in \C^{t_1+1}$ such that
$$ \Gamma\not\subset P_1^*$$
for all irreducible components $\Gamma$ of dimension $n-1$ of $P_0^*\cap V$, where
$P_1=\sum_{j=0}^{t_1}c_{1j}Q_j.$
This implies that $\dim \left(P_0^*\cap P_1^*\right)\cap V\le n-2.$

Step 2. For each irreducible component $\Gamma'$ of dimension $n-2$ of $\left(P_0^*\cap P_1^*\right)\cap V$, put 
$$V_{2\Gamma'}=\{c=(c_0,\ldots,c_{t_2})\in \C^{t_2+1}\ ;\ \Gamma\subset Q_c^*,\text{ where }Q_c=\sum_{j=0}^{t_2}c_jQ_j\}.$$
Hence, $V_{2\Gamma'}$ is a subspace of $\C^{t_2+1}$. Since $\dim \left(\bigcap_{i=0}^{t_2}Q_i^*\right)\cap V\le n-3$, there exists $i, (0\le i\le t_2)$ such that $\Gamma'\not\subset Q_i^*$. Then, $V_{2\Gamma'}$ is a proper subspace of $\C^{t_2+1}$. Since the set of irreducible components of dimension $n-2$ of $\left(P_0^*\cap P_1^*\right)\cap V$ is at most countable, 
$$ \C^{t_2+1}\setminus\bigcup_{\Gamma'}V_{2\Gamma'}\ne\varnothing. $$
Therefore, there exists $(c_{20},c_{21},\ldots,c_{2t_2})\in \C^{t_2+1}$ such that
$$ \Gamma'\not\subset P_3^*$$
for all irreducible components of dimension $n-2$ of $P_0^*\cap P_1^*\cap V$, where
$P_3=\sum_{j=0}^{t_2}c_{2j}Q_j.$
It implies that $\dim \left(P_1^*\cap P_2^*\cap P_3^*\right)\cap V\le n-3.$

Repeating again the above steps, after the $n^{\rm th}$-step we get hypersurfaces $P_0,\ldots,P_n$ satisfying
$$ \dim\left(\bigcap_{j=0}^tP_j^*\right)\cap V\le n-t-1\ (0\le t\le n). $$
In particular, $\left(\bigcap_{j=0}^{n}P_j^*\right)\cap V=\varnothing.$ We complete the proof of the lemma.
\end{proof}

\begin{lemma}
\label{3.1}
Let $t_0,t_1,\ldots,t_n$ be $n+1$ integers such that $1=t_0<t_1<\cdots <t_n$, and let $\Delta =\underset{1\le s\le n}\max\dfrac{t_s-t_0}{s}$. 
Then for every $n$ real numbers $a_0,a_1,\ldots,a_{n-1}$  with $a_0\ge a_1\ge\cdots\ge a_{n-1}\ge 1$, we have
$$ a_0^{t_1-t_0}a_1^{t_2-t_1}\cdots a_{n-1}^{t_{n}-t_{n-1}}\le (a_0a_1\cdots a_{n-1})^{\Delta}.$$
\end{lemma}
\begin{proof}
Let $s$ be an index, $1\le s\le n$ such that $\Delta=\dfrac{t_s-t_0}{s}$. 
\begin{itemize}
\item For every $1\le k\le s-1$, we have $\Delta=\dfrac{t_s-t_0}{s}\ge\dfrac{t_{s-k}-t_0}{s-k}$. This implies that 
$$\Delta=\dfrac{t_s-t_0}{s}\le\dfrac{t_s-t_0-(t_{s-k}-t_0)}{s-(s-k)}=\dfrac{t_s-t_{s-k}}{k},\text{ i.e., }t_s-t_{s-k}\ge k\Delta.$$
\item 
Similarly, for every $1\le k\le n-s$, we have $\Delta=\dfrac{t_s-t_0}{s}\ge\dfrac{t_{s+k}-t_0}{s+k}$. Then
$$\Delta=\dfrac{t_s-t_0}{s}\ge\dfrac{t_{s+k}-t_0-(t_s-t_0)}{(s+k)-s}=\dfrac{t_{s+k}-t_k}{k},\text{ i.e., }t_{s+k}-t_s\le k\Delta.$$
\end{itemize}
We set $m_n=\Delta$ and define
\begin{align*}
m_{n-1}&=t_{n}-t_{n-1}+\max\{0,m_n-\Delta\}, \\
m_{n-2}&=t_{n-1}-t_{n-2}+\max\{0,m_{n-1}-\Delta\},\\
\ldots&\ldots\\
m_{0}&=t_{1}-t_{0}+\max\{0,m_1-\Delta\}.
\end{align*}

We see that the set $\{i| s\le i\le n, m_i\le\Delta\}$ is not empty since $m_n=\Delta$. We set 
$$u=\min\{i| s\le i\le n, m_i\le\Delta\}.$$ 
Suppose that $u>s$ then $m_j>\Delta\ (\forall s\le j\le u-1)$, and hence we have the following estimate:
\begin{align*}
m_{u-1}&=t_u-t_{u-1}>\Delta,\\ 
m_{u-2}&=t_{u-1}-t_{u-2}+m_{u-1}-\Delta=t_u-t_{u-2}-\Delta>\Delta,\\
m_{u-3}&=t_{u-2}-t_{u-3}+m_{u-2}-\Delta=t_u-t_{u-3}-2\Delta>\Delta,\\ 
\ldots&\ldots\\
m_{s}&=t_{s+1}-t_{s}+m_{s+1}-\Delta=t_u-t_s-(u-s-1)\Delta>\Delta.
\end{align*}
This implies that
$$ t_u-t_s>(u-s)\Delta. $$
This is a contradiction. Therefore $u=s$, and hence $m_s\le \Delta$.

We also have the following estimate:
\begin{align*}
m_{s-1}&=t_{s}-t_{s-1}\ge\Delta,\\
m_{s-2}&=t_{s-1}-t_{s-2}+m_{s-1}-\Delta=t_s-t_{s-2}-\Delta\ge 2\Delta-\Delta=\Delta,\\
m_{s-3}&=t_{s-2}-t_{s-3}+m_{s-2}-\Delta=t_s-t_{s-3}-2\Delta\ge 3\Delta-2\Delta=\Delta,\\
\ldots&\ldots\\
m_{1}&=t_{2}-t_{1}+m_{2}-\Delta=t_s-t_{1}-(s-2)\Delta\ge (s-1)\Delta-(s-2)\Delta=\Delta,\\
m_{0}&=t_{1}-t_{0}+m_{1}-\Delta=t_s-t_{0}-(s-1)\Delta= s\Delta-(s-1)\Delta=\Delta.
\end{align*}
On the other hand, for each $i\in\{0,\ldots,n-2\}$ we have
$$\dfrac{a_i^{m_i}a_{i+1}^{\Delta}}{a_i^{t_{i+1}-t_i}a_{i+1}^{m_{i+1}}}=\dfrac{a_{i}^{\max\{0,m_{i+1}-\Delta\}}}{a_{i+1}^{m_{i+1}-\Delta}}\ge\left(\dfrac{a_i}{a_{i+1}}\right)^{\max\{0,m_{i+1}-\Delta\}}\ge 1,$$
i.e.,
$$a_i^{t_{i+1}-t_i}a_{i+1}^{m_{i+1}}\le a_i^{m_i}a_{i+1}^{\Delta}.$$
Then, we easily have that
\begin{align*}
a_0^{t_1-t_0}a_1^{t_2-t_1}\cdots a_{n-1}^{t_{n}-t_{n-1}}&= a_0^{t_1-t_0}a_1^{t_2-t_1}\cdots a_{n-2}^{t_{n-1}-t_{n-2}}a_{n-1}^{m_{n-1}}\\ 
&\le a_0^{t_1-t_0}a_1^{t_2-t_1}\cdots a_{n-2}^{t_{n-2}-t_{n-3}}a_{n-2}^{m_{n-2}}a_{n-1}^{\Delta}\\
&\le a_0^{t_1-t_0}a_1^{t_2-t_1}\cdots a_{n-2}^{t_{n-3}-t_{n-4}}a_{n-3}^{m_{n-3}}a_{n-2}^{\Delta}a_{n-1}^{\Delta}\\
&\ldots\\
&\le a_0^{t_1-t_0}a_1^{m_1}a_2^\Delta\cdots a_{n-1}^{\Delta}\\
&\le a_0^{m_0}a_1^{\Delta}a_2^\Delta\cdots a_{n-1}^{\Delta}\\
&=(a_0a_1\cdots a_{n-1})^{\Delta}.
\end{align*}
The lemma is proved.
\end{proof}

\section{General second main theorems for moving hypersurfaces}

In oder to prove the main theorem, we will first prove a general form of second main theorem for moving hyperplanes.

For a subset $\Psi\subset \mathcal{M}$ we denote by $\mathcal L(\Psi)$ the $\C$-vector space spanned by $\Psi$ over $\C$.
Assume that $q:=\sharp\Psi<\infty$, and $1 \in \Psi$. Then for a positive integer $p$, we set
$\Psi(p)=\{\varphi_1\varphi_2\cdots\varphi_p |\varphi_j\in\Psi;j=1,\ldots , p \}$. Then
$$1\in\Psi(p),\quad \Psi({p}) \subset \Psi({p+1}), \quad\sharp\Psi(p)=\binom{p+q-1}{p}=\binom{p+q-1}{q-1}.$$
Let $0< \epsilon<1$ be arbitrarily given. Then there exists a smallest integer, always denoted by $p$ in this section, such that
$$ \dfrac{\dim \mathcal L(\Psi(p+1))}{\dim \mathcal L (\Psi(p))}\leq (1+\epsilon).$$
Remark: In \cite{QT10}, we have the following estimate:
\begin{align}\label{4.1}
\dim\mathcal L(\Psi(p+1))\le \binom{p+q-1}{q-1}\leq\bigl [(1+\epsilon )^{[\frac{q}{\log^2(1+\epsilon)}]+1}\bigl ].
\end{align}
 
\begin{theorem}\label{4.2}
Let $f$ be a linearly nondegenerate meromorphic mapping of $\C^m$ in $\P^N(\C)$ with a reduced representation ${\bf f}=(f_0,...,f_N)$ and let $H_1,..,H_q$ be $q$ arbitrary slowly (with respect to $f$) moving hyperplanes in $\P^N(\C)$. Then for every $\epsilon >0,$
$$ \biggl\|\ \int_{S(r)}\max_{K}\log\left (\prod_{j\in K}\frac{\|{\bf f}\|}{|H_j({\bf f})|}\right)\sigma_{m}\le (N+1+\epsilon)T_f(r),$$
where the maximum is taken over all subsets $K\subset\{1,...,q\}$ such that $\{H_j\ ;\ j\in K\}$ is linearly independent over the field $\mathcal M$.
\end{theorem}
The version of this theorem for fixed hyperplanes is firstly proved by M. Ru in \cite[Theorem 2.3]{Ru97}. For our purpose in this paper, we need this version for moving hyperplanes.
\begin{proof}
By adding more fixed hyperplanes if necessary, we will prove that 
$$ \biggl\|\ \int_{S(r)}\max_{K}\log\left (\prod_{j\in K}\frac{\|{\bf f}\|}{|H_j({\bf f})|}\right)\sigma_{m}\le (N+1+\epsilon)T_f(r),$$
where the maximum is taken over all subsets $K\subset\{1,...,q\}$ such that $\sharp K=N+1$ and $\{H_j| j\in K\}$ is linearly independent over the field $\mathcal M$ of all meromorphic functions on $\C^m$.

Without loss of generality we may assume that
$$ H_i(z)(x_0,\ldots,x_N)=\sum_{j=0}^Na_{ij}(z)x_j,$$
where all $a_{ij}$ are small functions with respect to $f$ and $a_{i0}=1$.
We put $\Psi=\{a_{ij}\}$. By Lemma \ref{3.3} there exists a positive integer $p$ such that 
$$\frac{\dim \mathcal L(\Psi(p+1))}{\dim \mathcal L(\Psi(p))}\leq 1+\dfrac{\epsilon}{N+1}.$$

We put
$$
s=\dim \mathcal{L}(\Psi(p)),\qquad
t=\dim \mathcal{L}(\Psi(p)+1)).
$$ 
Let $\{b_1, \ldots ,b_s\}$ be a base of $\mathcal L(\Psi(p))$ and $\{b_1, \ldots ,b_t\}$ be a base of $\mathcal L(\Psi(p+1))$. Then $\dfrac{t}{s}\leq 1+\epsilon$ and
$\{b_jf_k(1\leq i\leq t,0\leq k\leq N)\}$ is linearly independent over $\C$.

Let $K=\{i_1,\ldots,i_{N+1}\}$ be a set such that $\{H_j|j\in K\}$ is linearly independent over $\mathcal M$, where $1\le i_1<\cdots<i_{N+1}\le q$. It is easy to see that $\{b_jH_{i_v}({\bf f}) (1\leq j\leq s,1\le v\le N+1)\}$ is linearly independent over  $\C$. Then we may choose  $\beta ^{kl}_{mj}\in \C$ such that there is  $C_{K}\in GL((N+1)t;\C)$ such that
\begin{align*}
&\det (b_jH_{i_v}({\bf f}) (1\leq j\leq s,1\leq v\leq N+1), h_{ul}(s+1\leq l \leq t,0\leq u\leq N))\\
&=C_{K}\det (b_jf_k(1\leq j\leq t,0\leq k\leq N)),
\end{align*}
where $h_{ul}=\sum_{1\leq k\leq t,0\leq l'\leq N}\beta^{kl'}_{uj}b_kf_{l'}(s+1\leq j\leq t,0\leq u\leq N)$, and $C_{K}$ is a nonzero constant.

Let  $\alpha:=(\alpha_1, \ldots ,\alpha_{(N+1)t})\in (\Z_+^m)^{(N+1)t}$
be the admissible set such that 
$$W\equiv\det\bigl (\mathcal D^{\alpha_w}b_jf_k(1\leq j\leq t,0\leq k\leq N)\bigl )_{1\leq w\leq (N+1)t}\not\equiv 0.$$
Note that $|\alpha_i|\leq (N+1)t-1,\forall 1\leq i\leq (N+1)t$. 
Set
$$W_{K}\equiv\det\bigl (\mathcal D^{\alpha_w}b_jH_{i_v}({\bf f}),\mathcal D^{\alpha_w}h_{ul}\bigl ),$$
where $1\leq j\leq t, 1\leq v\leq N+1$, $s+1\leq l \leq t, 0\leq u\leq N$, and $1\leq w\leq (N+1)t$.
It is easy to see that $W_{K}=C_K\cdot W$. 

This implies that, for $K=\{i_1,\ldots,i_{N+1}\}$,
\begin{align*}
s\log\left (\prod_{j\in K}\frac{\|{\bf f}\|}{|H_j({\bf f})|}\right)&\le (N+1)t\log \|{\bf f}\|+\log\left (\frac{1}{(\prod_{j\in K}|H_j({\bf f})|)^s\|{\bf f}\|^{(N+1)(t-s)}}\right)\\
&\le (N+1)\log \|{\bf f}\|+\sum_{K}\log^{+}\left (\frac{|W_K|}{(\prod_{j\in K}|H_j({\bf f})|)^s\|{\bf f}\|^{(N+1)(t-s)}}\right)\\
&\quad\quad -\log|W|+O(1).
\end{align*}
Integrating both sides of the above of the above inequality over $S(r)$ and applying the lemma on logarithmic derivative, we get
\begin{align*}
\biggl\|\ \int_{S(r)}\max_{K}\log\left (\prod_{j\in K}\frac{\|{\bf f}\|}{|H_j({\bf f})|}\right)\sigma_{m}&\le\dfrac{(N+1)t}{s}T_f(r)-\dfrac{1}{s}N_{W}(r)+o(T_f(r))\\
&\le (N+1+\epsilon)T_f(r).
\end{align*}
The theorem is proved.
\end{proof}

\begin{proof}[{\bf Proof of Theorem \ref{1.1}}]

Take $z_0$ be a point such that all coefficients of $R_1,\ldots,R_\lambda$ are holomorphic at $z_0$ and 
$$ \Delta_f =\max_{\Gamma\subset\{1,\ldots,q\}}\dfrac{\sharp\Gamma}{\dim\bigcap_{i=1}^\lambda R_i(z_0)^*-\dim\bigcap_{i=1}^\lambda R_i(z_0)^*\cap\bigcap_{j\in\Gamma}Q_i(z_0)^*},$$
where $\{R_1,\ldots,R_\lambda\}$ is a minimal subset of $\mathcal K[x_0,\ldots,x_N]$, which generates $I(f(\C^m))$. 
Let $V_z=\bigcap_{i=1}^\lambda R_i(z)^*$, which is a projective subvarieties of $\P^N(\C)$ of dimension $n_f$ for generic points $z\in\C^m$. We may suppose that $\dim V_{z_0}=n_f$. It is suffice for us to consider the case where $\Delta_f<\dfrac{q}{n_f+1}$. Note that $\Delta_f\ge 1$, and hence $q>n_f+1$. If there exists $i\in\{1,\ldots,q\}$ such that $\bigcap_{\underset{j\ne i}{j=1}}^qQ_j(z_0)^*\cap V_{z_0}\ne\varnothing$ then 
$$ \Delta_f\ge\dfrac{q-1}{n_f}>\dfrac{q}{n_f+1}.$$
This is a contradiction. Therefore, $\bigcap_{\underset{j\ne i}{j=1}}^qQ_j(z_0)^*\cap V_{z_0}=\varnothing$ and hence $\bigcap_{\underset{j\ne i}{j=1}}^qQ_j(z)^*\cap V_{z}=\varnothing$ generically, for all $i\in\{1,2,\ldots,q\}$.

Since the number of slowly moving hypersurfaces occurring in this proof is finite, we may choose a function $c\in\mathcal C_{\mathcal K}$ such that for each given slowly moving hypersurface $Q$ in this proof, we have
$$ Q(z)({\bf x})\le c(z)\|{\bf x}\|^{\deg Q} $$
for all ${\bf x}=(x_0,\ldots,x_N)\in\C^{N+1}$, $z\in\C^m$. 

By usual argument, it is suffice for us to prove the theorem only for the case where all hypersurfaces $Q_i\ (1\le i\le q)$ are of the same degree $d$. 
We denote by $\mathcal I$ the set of all bijections from $\{0,\ldots,q-1\}$ into $\{1,\ldots,q\}$. Denote by $n_0$ the cardinality of $\mathcal I$, $n_0=q!$, and we write
$\mathcal I=\{I_1,\ldots,I_{n_0}\}$,
where $I_i=(I_i(0),\ldots,I_i(q-1))\in\mathbb N^q$ and $I_1<I_2<\cdots <I_{n_0}$ in the lexicographic order.

For each $I_i\in\mathcal I$, since $\bigcap_{j=0}^{q-2}Q_{I_i(j)}(z_0)^*\cap V_{z_0}=\varnothing$, there exist $n_f+1$ integers $t_{i,0},t_{i,1},\ldots,t_{i,n_f}$ with $0=t_{i,0}<\cdots<t_{i,n_f}=l_i$, where $l_i\le q-2$ such that $\bigcap_{j=0}^{l_i}Q_{I_i(j)}(z_0)^*\cap V_{z_0}=\varnothing$ and
$$\dim\left (\bigcap_{j=0}^{s}Q_{I_i(j)}(z_0)^*\right )\cap V_{z_0}=n_f-u\ \forall t_{i,u-1}\le s<t_{i,u},1\le u\le n_f.$$
Then, $\Delta_f \ge\dfrac{t_{i,u}-t_{i,0}}{u}$ for all $1\le u\le n_f.$ Denote by $P'_{i,0},\ldots,P'_{i,n_f}$ the hypersurfaces obtained in Lemma \ref{3.2} with respect to the hypersurfaces $Q_{I_i(0)}(z_0),\ldots,Q_{I_i(l_i)}(z_0)$. Now, for each $P'_{i,j}$ constructed by
$$ P'_{i,j}=\sum_{s=0}^{t_{i,j}}a_{i,j,s}Q_{I_i(s)}(z_0)\ (a_{i,j,s}\in\C) $$
we define
$$ P_{i,j}(z)=\sum_{s=0}^{t_{i,j}}a_{i,j,s}Q_{I_i(s)}(z).$$
Hence $\{P_{i,j}\}_{j=0}^{n_f}$ is a family of moving hypersurfaces in $\P^N(\C)$ with $P_{i,j}(z_0)=P'_{i,j}$. Then $\bigcap_{j=0}^{n_f}P_{i,j}(z_0)^*\cap V_{z_0}=\varnothing$, and hence $\{P_{i,j}(z)\}_{j=0}^{n_f}$ is in general position with respect to $V_z$ for generic points $z\in\C^m$. 
We may choose a positive constant $B\ge 1$, commonly for all $I_i\in\mathcal I$, such that
$$ |P_{i,j}({\bf x})|\le B\max_{0\le s\le t_{i,j}}|Q_{I_i(j)}({\bf x})|, $$
for all $0\le j\le N$ and for all ${\bf x}=(x_0,\ldots,x_N)\in\C^{N+1}$. 
Denote by $\mathcal S$ the set of all points $z\in\C^m$ such that $\bigcap_{j=0}^{n_f}P_{i,j}(z)^*\cap V_z\ne\varnothing$ for some $I_i$. Then $\mathcal S$ is a proper analytic subset of $\C^m$.

By Lemma \ref{new2}, there exists a function $A\in\mathcal C_{\mathcal K}$, which is chosen common for all $I_i$, such that
$$ \|{\bf f} (z)\|^d\le A(z)\max_{0\le j\le l_i}|Q_{I_i(j)}(z)({\bf f}(z))|\ (\forall I_i\in\mathcal I).$$
Fix an element $I_i\in\mathcal I$. Denote by $S(i)$ the set of all points 
$$z\in \C\setminus\left\{\bigcup_{i=1}^qQ_i(z)({\bf f}(z))^{-1}(\{0\})\cup \bigcup_{\underset{I_i\in\mathcal I}{0\le j\le n}}P_{i,j}(z)({\bf f}(z))^{-1}(\{0\})\right\}$$
such that
$$ |Q_{I_i(0)}(z)({\bf f}(z))|\le |Q_{I_i(1)}(z)({\bf f}(z))|\le\cdots\le |Q_{I_i(q-1)}(z)({\bf f}(z))|.$$
Therefore, for generic points $z\in S(i)$, By Lemma \ref{3.1} we have
\begin{align*}
\prod_{i=1}^q\dfrac{\|{\bf f} (z)\|^d}{|Q_i(z)({\bf f}(z))|}&\le\dfrac{A(z)^{q-l_i}}{c(z)^{l_j}}\prod_{j=0}^{l_j-1}\dfrac{c(z)\|{\bf f} (z)\|^d}{|Q_{I_i(j)}(z)({\bf f}(z))|}\\
&\le\dfrac{A(z)^{q-l_i}}{c(z)^{l_j}}\prod_{j=0}^{n_f-1}\left(\dfrac{c(z)\|{\bf f} (z)\|^d}{|Q_{I_i(t_j)}(z)({\bf f}(z))|}\right)^{t_{i,j+1}-t_{i,j}}\\
&\le \dfrac{A(z)^{q-l_i}}{c(z)^{l_j}}\prod_{j=0}^{n_f-1}\left(\dfrac{c(z)\|{\bf f} (z)\|^d}{|Q_{I_i(t_j)}(z)({\bf f}(z))|}\right)^{\Delta_f}\\
&\le \dfrac{A(z)^{q-l_i}B(z)^{n_f\Delta_f}}{c(z)^{l_j-n_f\Delta_f}}\prod_{j=0}^{n_f-1}\left(\dfrac{\|{\bf f} (z)\|^d}{|P_{i,j}(z)({\bf f}(z))|}\right)^{\Delta_f}\\
&\le C(z)\prod_{j=0}^{n_f}\left(\dfrac{\|{\bf f} (z)\|^d}{|P_{i,j}(z)({\bf f}(z))|}\right)^{\Delta_f},
\end{align*}
where $C\in\mathcal C_{\mathcal K}$, chosen commonly for all $I_i\in\mathcal I$.

For $z\not\in\mathcal S$, consider the mapping $\Phi_z$ from $V_z$ into $\P^{l-1}(\C)\ (l=n_0(n_f+1))$, which maps a point ${\bf x}=(x_0:\cdots:x_N)\in V_z$ into the point $\Phi_z({\bf x})\in\P^{l-1}(\C)$ given by
$$\Phi_z({\bf x})=(P_{1,0}(z)(x):\cdots : P_{1,n_f}(z)(x):\cdots:P_{n_0,0}(z)(x):\cdots :P_{n_0,n_f}(z)(x)),$$
where $x=(x_0,\ldots,x_N)$. Set
$$\tilde\Phi_z(x)=(P_{1,0}(z)(x),\ldots, P_{1,n_f}(z)(x),\ldots,P_{n_0,0}(z)(x),\ldots,P_{n_0,n_f}(z)(x)).$$
Let $Y_z=\Phi_z(V_z)$. Since $V_z\cap\bigcap_{j=0}^{n_f}P_{1,j}(z)^*=\varnothing$, $\Phi_z$ is a finite morphism on $V_z$ and $Y_z$ is a complex projective subvariety of $\P^{l-1}(\C)$ with $\dim Y_z=n_f$ and of degree
$$\delta_z:=\deg Y_z\le d^{n_f}.\deg V_z\le d^{n_f}\delta_f\ (\text{for generic points }z).$$ 
For every 
$${\bf a} = (a_{1,0},\ldots,a_{1,n_f},a_{2,0}\ldots,a_{2,n_f},\ldots,a_{n_0,0},\ldots,a_{n_0,n_f})\in\mathbb Z^l_{\ge 0}$$ 
and
$${\bf y} = (y_{1,0},\ldots,y_{1,n_f},y_{2,0}\ldots,y_{2,n_f},\ldots,y_{n_0,0},\ldots,y_{n_0,n_f})$$ 
we denote ${\bf y}^{\bf a} = y_{1,0}^{a_{1,0}}\ldots y_{1,n_f}^{a_{1,n_f}}\ldots y_{n_0,0}^{a_{n_0,0}}\ldots y_{n_0,n_f}^{a_{n_0,n_f}}$. Let $u$ be a positive integer. We set
\begin{align*}
\xi_u:=\binom{l+u-1}{u}-1,
\end{align*}
and define
$$ Y_{z,u}:=\C[y_1,\ldots,y_l]_u/(I_{Y_z})_u, $$
which is a vector space of dimension $H_{Y_z}(u)$. 

Denote by $I(Y)_u$ the subspace of the $\mathcal K$-vector space $\mathcal K[y_1,\ldots,y_l]_u$ consisting of all homogeneous polynomials $P\in \mathcal K[y_1,\ldots,y_l]_u$ (including the zero polynomial) such that 
$$P(z)(\Phi_z(f(z)))\equiv 0.$$
Let $(\tilde R_1,\ldots,\tilde R_p)$ be an ordered $\mathcal K$-basis of $I(Y)_u$. By enlarging the analytic set $\mathcal S$ if necessary, we may assume that all zeros and poles of all nonzero coefficients of $\tilde R_i\ (1\le i\le p)$ are contained in $\mathcal S$, also all above assertions for generic points $z\in\C^m$ will hold for all $z\not\in\mathcal S$. Choose $\xi_u-p$ nonzero monic homogeneous monomial $v_1,\ldots,v_{\xi_u-p}$ of degree of $u$ in variables $y_1,\ldots,y_l$  such that $\{\tilde R_1,\ldots,\tilde R_p,v_1,\ldots,v_{\xi_u-p}\}$ is a $\mathcal K$-basis of $\mathcal K[y_1,\ldots,y_l]_u$. 

\begin{claim}\label{4.3}
There is a proper analytic subset of $\C^m$ such that for all $z$ outside this set, the family of equivalent classes of $v_1,\ldots,v_{\xi_u-p}$ is a basis of $Y_{z,u}$ and the family $\{\tilde R_1(z),\ldots,\tilde R_p(z)\}$ is a basis of $(I_{Y_z})_u$.
\end{claim}
Indeed, set $ V_j(z)(x)=v_j(\tilde\Phi_z(x)),$ which is a homogeneous polynomial in $\mathcal K[x_0,\ldots,x_N]_{ud}$. We see that $V_1,\ldots,V_{\xi_u-p}$ are linearly independent over $\mathcal K$. There exists $p'$ nonzero monic homogeneous monomial $W_1,\ldots,W_{p'}$ in $\C[x_0,\ldots,x_N]_{du}$ such that $V_1,\ldots,V_{\xi_u-p},W_1,\ldots,W_{p'}$ is an $\mathcal K$-basis of $\mathcal K[x_0,\ldots,x_N]_{ud}$. We call $\{W_1,\ldots,W_{p'},W_{p'+1},\ldots,W_{\xi_u-p+p'}\}$ the set of all nonzero monic homogeneous monomials in $\C[x_0,\ldots,x_N]_{du}$. Hence, for each $p'<i\le \xi_u-p+p'$, we have
$$ W_i=\sum_{j=1}^{\xi_i-p}w_{ij}V_j+\sum_{j=1}^{p'}w'_{ij}W_j,\  (w_{ij},w'_{ij}\in\mathcal K).$$
By enlarging again $S$ if necessary, we may suppose that all zeros and poles of all nozero coefficients of $w_{ij}$ and $w'_{ij}$ belong to $\mathcal S$. Therefore, for every $z\not\in\mathcal S$, the family 
$$\{V_1(z),\ldots,V_{\xi_u-p}(z),W_1,\ldots,W_{p'}\}$$
 generates the family 
$$\{W_1,\ldots,W_{p'},W_{p'+1},\ldots,W_{\xi_u-p+p'}\},$$
 and hence it is a basis of $\C[x_0,\ldots,x_N]_{du}$. This implies that $\{v_j(\tilde\Phi_z)\}_{j=1}^{\xi_u-p}$ is $\C$-linearly independent, and hence the equivalence classes of $v_1,\ldots,v_{\xi_u-p}$ modulo $(I_{Y_z})_u$ are $\C$-linear independent in $Y_{z,u}$ for all $z\not\in\mathcal S$. 

Denote by $\mathcal T=\{T_1,\ldots,T_{\xi_u}\}$ the set of all nonzero monic homogeneous monomial of degree of $u$ in variables $y_1,\ldots,y_l$. Then $\{T_1,\ldots,T_{\xi_u}\}$ is an $\mathcal K$-basis of $\mathcal K[y_1,\ldots,y_l]_u$, and also be an $\C$-basis of $\C[y_1,\ldots,y_l]_u$.

Then for each $1\le j\le\xi_u$, we have
$$ T_j=\sum_{i=1}^{p}b_{ji}\tilde R_i+\sum_{i=1}^{\xi_u-p}b'_{ji}v_i,$$
where $b_{ji},b_{ji}'\in\mathcal K$. By enlarging $\mathcal S$ if necessary, we may assume that all zeros and poles of all nonzero functions $b_{ji}$ and $b'_{ji}$ are contained in $\mathcal S$. Therefore, for all $z\not\in\mathcal S$, $\{\tilde R_1(z),\ldots,\tilde R_p(z),v_1,\ldots,v_{\xi_u-p}\}$ generates the $\C$-basis $\{T_1,\ldots,T_{\xi_u}\}$ of $\C[y_1,\ldots,y_l]_u$, and hence it is also $\C$-basis of $\C[y_1,\ldots,y_l]_u$. In particular $\tilde R_1(z),\ldots,\tilde R_p(z)$ are linearly independent over $\C$ for all $z\not\in\mathcal S$.

Therefore, for $z\not\in\mathcal S$, we have
\begin{align*}
\xi_u=\dim \C[y_1,\ldots,y_l]_u&=\dim\dfrac{\C[y_1,\ldots,y_l]_u}{(I_{Y_z})_u}+\dim (I_{Y_z})_u\\
&\le \sharp\{v_1,\ldots,v_{\xi_u-p}\}+\sharp\{\tilde R_1(z),\ldots,\tilde R_p(z)\}=\xi_u.
\end{align*}
This yields that $\dim\dfrac{\C[y_1,\ldots,y_l]_u}{(I_{Y_z})_u}=\sharp\{v_1,\ldots,v_{\xi_u-p}\}$ and $\dim (I_{Y_z})_u=\sharp\{\tilde R_1(z),\ldots,\tilde R_p(z)\}$. Then the set of equivalent classes of $v_1,\ldots,v_{\xi_u-p}$ is a basis of $Y_{z,u}$, and $\{\tilde R_1(z),\ldots,\tilde R_p(z)\}$ are basis of $(I_{Y_z})_u$. The claim is proved.

From this claim, we see that
$$\xi_u-p=H_{Y_z}(u)\ \forall z\in\C^m\setminus\mathcal S.$$
On the other hand, with the same argument as in the claim, by enlarging $\mathcal S$ if necessary, we may assume that if a subset $\{v_1',\ldots,v'_{\xi_u-p}\}$ of $\mathcal T$ satisfies $\{\tilde R_1,\ldots,\tilde R_p,v_1',\ldots,v'_{\xi_u-p}\}$ is an $\mathcal K$-basis of $\mathcal K[y_1,\ldots,y_l]_u$ then the set of equivalent classes of $v'_1,\ldots,v'_{\xi_u-p}$ modulo $(I_{Y_z})_u$ is a basis of $Y_{z,u}\ \forall z\not\in\mathcal S$.

Now, for every subset $\{v_1',\ldots,v'_{\xi_u-p}\}$ of $\mathcal T$ such that $\{\tilde R_1,\ldots,\tilde R_p,v_1',\ldots,v'_{\xi_u-p}\}$ is linearly dependent over $\mathcal K$, we take a non-trivial linear combination 
$$ \sum_{i=1}^{p}c_i\tilde R_i+ \sum_{i=1}^{\xi_u-p}c'_iv'_i\equiv 0,$$
where $c_i,c'_i\in\mathcal K$, not all zeros. By enlarging $S$ again, we may assume that all poles and zeros of $c_i,c'_i$ belong to $\mathcal S$. Hence, for all $z\not\in\mathcal S$, $\{\tilde R_1(z),\ldots,\tilde R_p(z),v_1',\ldots,v'_{\xi_u-p}\}$ is linearly dependent over $\C$. Consequently, the set of equivalent classes of $v_1',\ldots,v'_{\xi_u-p}$ modulo $(I_{Y_z})_u$ is not a basis of $Y_{z,u}$.

Now, consider the meromorphic mapping $F$ into $\P^{\xi_u-p}(\C)$ with the representation
$$ {\bf F}=(v_0(\tilde\Phi\circ {\bf f}),\ldots,v_{\xi_u-p}(\tilde\Phi\circ {\bf f})). $$
Hence $F$ is linearly nondegenerate over $\mathcal K$, since $f$ is algebraically nondegenerate over $\mathcal K$.

Now, we fix an index $i\in\{1,\ldots,n_0\}$ and a point $z\in S(i)\setminus\mathcal S$. We define 
$${\bf c}_z = (c_{1,0,z},\ldots,c_{1,n_f,z},c_{2,0,z},\ldots,c_{2,n_f,z},\ldots,c_{n_0,0,z},\ldots,c_{n_0,n_f,z})\in\mathbb Z^{l},$$ 
where
\begin{align}\label{4.4}
c_{i,j,z}:=\log\frac{\|{\bf f}(z)\|^d\|P_{i,j}(z)\|}{|P_{i,j}(z)({\bf f}(z))|}\text{ for } i=1,\ldots,n_0 \text{ and }j=0,\ldots,n_f.
\end{align}
We see that $c_{i,j,z}\ge 0$ for all $i$ and $j$. By the definition of the Hilbert weight, there are ${\bf a}_{1,z},\ldots,{\bf a}_{\xi_u-p,z}\in\mathbb N^{l}$ with
$$ {\bf a}_{i,z}=(a_{i,1,0,z},\ldots,a_{i,1,n_f,z},\ldots,a_{i,n_0,0,z},\ldots,a_{i,n_0,n_f,z}), $$
where $a_{i,j,s,z}\in\{1,\ldots,\xi_u\},$ such that the residue classes modulo $(I_Y)_u$ of ${\bf y}^{{\bf a}_{1,z}},\ldots,{\bf y}^{{\bf a}_{\xi_u-p,z}}$ form a basic of $\C[y_1,\ldots,y_l]_u/(I_{Y_z})_u$ and
\begin{align}\label{4.5}
S_Y(u,{\bf c}_z)=\sum_{i=1}^{\xi_u-p}{\bf a}_{i,z}\cdot{\bf c}_z.
\end{align}
We see that ${\bf y}^{{\bf a}_{i,z}}\in\mathcal T$. Then as above $\{\tilde R_1(z),\ldots,\tilde R_p(z),{\bf y}^{{\bf a}_{1,z}},\ldots,{\bf y}^{{\bf a}_{\xi_u-p,z}}\}$ is a basis of $\mathcal K[y_1,\ldots,y_l]$. Therefore, $\{{\bf y}^{{\bf a}_{1,z}},\ldots,{\bf y}^{{\bf a}_{\xi_u-p,z}}\}$ is a basis of $\dfrac{\mathcal K[y_1,\ldots,y_l]_u}{I(Y)_u}$. Then, we may write
$$ {\bf y}^{{\bf a}_{i,z}}=L_{i,z}(v_0,\ldots,v_{H_Y(u)})\ \ \  \text{modulo }I(Y)_u, $$ 
where $L_{i,z}\ (1\le i\le \xi_u-p)$ are independent linear forms with coefficients in $\mathcal K$.
We have
\begin{align*}
\log\prod_{i=1}^{\xi_u-p} |L_{i,z}({\bf F}(z))|&=\log\prod_{i=1}^{\xi_u-p}\prod_{\overset{1\le t\le n_0}{0\le j\le n_f}}|P_{t,j}({\bf f}(z))|^{a_{i,t,j,z}}\\
&=-S_Y(u,{\bf c}_z)+du(\xi_u-p)\log \|{\bf f}(z)\| +O(u(\xi_u-p)).
\end{align*}
It follows that
\begin{align*}
\log\prod_{i=1}^{\xi_u-p}\dfrac{\|{\bf F}(z)\|\cdot \|L_{i,z}\|}{|L_{i,z}({\bf F}(z))|}=&S_Y(u,{\bf c}_z)-du(\xi_u-p)\log \|{\bf f}(z)\| \\
&+(\xi_u-p)\log \|{\bf F}(z)\|+O(u(\xi_u-p)).
\end{align*}
Note that the number of these linear forms $L_{i,z}$ is finite. Denote by $\mathcal L$ the set of all $L_{i,z}$ occurring in the above inequalities. We have
\begin{align}\label{4.6}
\begin{split}
S_Y(u,{\bf c}_z)\le&\max_{\mathcal J\subset\mathcal L}\log\prod_{L\in \mathcal J}\dfrac{\|{\bf F}(z)\|\cdot \|L\|}{|L({\bf f}(z))|}+du(\xi_u-p)\log \|{\bf f}(z)\|\\
& -(\xi_u-p)\log \|{\bf F}(z)\|+O(u(\xi_u-p)),
\end{split}
\end{align}
where the maximum is taken over all subsets $\mathcal J\subset\mathcal L$ with $\sharp\mathcal J=H_Y(u)$ and $\{L|L\in\mathcal J\}$ is linearly independent over $\mathcal K$.
From Theorem \ref{2.3} we have
\begin{align}\label{4.7}
\dfrac{1}{u(\xi_u-p)}S_{Y_z}(u,{\bf c}_z)\ge&\frac{1}{(n_f+1)\delta_z}e_{Y_z}({\bf c}_z)-\frac{(2n_f+1)\delta_z}{u}\max_{\underset{0\le j\le n_f}{1\le i\le n_0}}c_{i,j,z}
\end{align}
We chose an index $i_0$ such that $z\in S(i_0)$. It is clear that
\begin{align*}
\max_{\underset{0\le j\le n_f}{1\le i\le n_0}}c_{i,j,z}\le \sum_{0\le j\le n_f}\log\frac{\|{\bf f}(z)\|^d\|P_{i_0,j}\|}{|P_{i_0,j}(z)({\bf f}(z))|}+O(1),
\end{align*}
where the term $O(1)$ does not depend on $z$ and $i_0$.
Combining (\ref{4.6}), (\ref{4.7}) and the above remark, we get
\begin{align}\nonumber
\frac{1}{(n_f+1)\delta_z}e_{Y_z}({\bf c}_z)\le &\dfrac{1}{u(\xi_u-p)}\left (\max_{\mathcal J\subset\mathcal L}\log\prod_{L\in \mathcal J}\dfrac{\|{\bf F}(z)\|\cdot \|L\|}{|L({\bf F}(z))|}-(\xi_u-p)\log \|{\bf F}(z)\|\right )\\
\label{4.8}
\begin{split}
&+d\log \|{\bf F}(z)\|+\frac{(2n_f+1)\delta_z}{u}\max_{\underset{0\le j\le n_f}{1\le i\le n_0}}c_{i,j,z}+O(1/u)\\
\le &\dfrac{1}{u(\xi_u-p)}\left (\max_{\mathcal J\subset\mathcal L}\prod_{L\in\mathcal J}\dfrac{\|{\bf F}(z)\|\cdot \|L\|}{|L({\bf F}(z))|}-(\xi_u-p)\log \|{\bf F}(z)\|\right )\\
&+d\log \|{\bf F}(z)\|+\frac{(2n_f+1)\delta_z}{u}\sum_{0\le j\le n_f}\log\frac{\|{\bf f}(z)\|^d\|P_{i_0,j}\|}{|P_{i_0,j}(z)({\bf f}(z))|}+O(1/u).
\end{split}
\end{align}
Since $\{P_{i_0,0}(z)=\cdots=P_{i_0,n_f}(z)=0\}\cap V_z=\varnothing$, by Lemma \ref{2.4}, we have
\begin{align}\label{4.9}
e_{Y_z}({\bf c}_z)\ge (c_{i_0,0,z}+\cdots +c_{i_0,n_f,z})\cdot\delta_z =\left (\sum_{0\le j\le n_f}\log\frac{\|{\bf f}(z)\|^d\|P_{i_0,j}(z)\|}{|P_{i_0,j}(z)({\bf f}(z))|}\right )\cdot\delta_z,
\end{align}
for all $z\not\in\mathcal S$.

Then, from (\ref{4.4}), (\ref{4.8}) and (\ref{4.9}) we have
\begin{align*}
\frac{1}{\Delta_f}\log \prod_{i=1}^q\dfrac{\|{\bf f} (z)\|^d}{|Q_i(z)({\bf f}(z))|}&\le\dfrac{n_f+1}{u(\xi_u-p)}\left (\max_{\mathcal J\subset\mathcal L}\log\prod_{L\in\mathcal J}\dfrac{\|{\bf F}(z)\|\cdot \|L\|}{|L({\bf F}(z))|}-(\xi_u-p)\log \|{\bf F}(z)\|\right )\\
&+\frac{(2n_f+1)(n_f+1)\delta_z}{u}\sum_{\underset{0\le j\le n_f}{1\le i\le n_0}}\log\frac{\|{\bf f}(z)\|^d\|P_{i,j}(z)\|}{|P_{i,j}(z)({\bf f}(z))|}\\
&+d(n_f+1)\log \|{\bf f}(z)\|+\frac{1}{\Delta_f}\log C(z)+O(1)
\end{align*}
for generic points $z\in\C^m$, where the term $O(1)$ does not depend on $z$. 
Integrating both sides of the above inequality, we obtain 
\begin{align}\nonumber
\frac{1}{d}\sum_{i=1}^qm_f(r,Q_i)\le& \dfrac{\Delta_f(n_f+1)}{du(\xi_u-p)}\left (\int\limits_{S(r)}\max_{\mathcal J\subset\mathcal L}\log\prod_{L\in\mathcal J}\dfrac{\|{\bf F}(z)\|\cdot \|L\|}{|L({\bf F}(z))|}\sigma_m-(\xi_u-p)T_F(r)\right )\\
\label{4.10}
&+\Delta_f(n_f+1)T_f(r)+\frac{\Delta_f(2n_f+1)(n_f+1)\delta_f}{du}\sum_{{\underset{0\le j\le n_f}{1\le i\le n_0}}}m_f(r,P_{i,j})
\end{align} 
(note that $\delta_z\le d^{n_f}\delta_f$ for all $z\not\in\mathcal S$). On the other hand, by Theorem \ref{4.2}, we have
\begin{align}
\label{4.11}
\begin{split}
\biggl \|\ \int\limits_{S(r)}\max_{\mathcal J\subset\mathcal L}&\log\prod_{L\in\mathcal J}\dfrac{\|{\bf F}(z)\|\cdot \|L\|}{|L({\bf F}(z))|}\sigma_m-(\xi_u-p)T_F(r)\\
&\le \dfrac{\epsilon(\xi_u-p)}{2\Delta_f(n_f+1)}T_F(r)\le \dfrac{du\epsilon(\xi_u-p)}{2\Delta_f(n_f+1)}T_f(r).
\end{split}
\end{align}
Combining (\ref{4.10}) with (\ref{4.11}), we have
\begin{align*}
\bigl\|\ (q-\Delta_f(n_f+1)\bigl )T_f(r)&\le\sum_{i=1}^q\frac{1}{d}N_{Q_i({\bf f})}(r)+\frac{\Delta_f(2n_f+1)(n_f+1)\delta_f}{ud}\sum_{{\underset{0\le j\le n_f}{1\le i\le n_0}}}m_f(r,P_{i,j})\\
&\ \ \ +\dfrac{\epsilon}{2}T_f(r)\\
&\le\sum_{i=1}^q\frac{1}{d}N_{Q_i({\bf f})}(r)+\epsilon T_f(r)\ \  (\text{ by choosing $u$ large enough}).
\end{align*}
The theorem is proved.
\end{proof}

\begin{remark}
{\rm (a) By the proof of Theorem \ref{4.2}, the inequality (\ref{4.11}) can be replaced by
$$ \biggl \|\ \int\limits_{S(r)}\max_{\mathcal J\subset\mathcal L}\log\prod_{L\in\mathcal J}\dfrac{\|{\bf F}(z)\|\cdot \|L\|}{|L({\bf F}(z))|}\sigma_m-(\xi_u-p)T_F(r)\le \dfrac{1}{s}N_{W(F)}(r)+\dfrac{\epsilon(\xi_u-p)}{2\Delta_f(n_f+1)}T_F(r). $$
where $W(F)$ is a general Wronskian of $\{b_{j}v_i(\tilde\Phi\circ {\bf f}); 1\le j\le t, 0\le i\le n_u\}$; $\{b_1,\ldots,b_t\}$ is a basis of $\Psi (q'+1)$ with a positive integer $q'$, $\Psi$ is the set of all coefficients of $Q_i\ (1\le i\le q)$ and $s=\dim\mathcal L(\Psi(q'))$.
Then, the second main theorem can be obtained as follows:
$$\|\ (q-\Delta_f(n_f+1)-\epsilon\bigl )T_f(r)\le\sum_{i=1}^q\frac{1}{d}N_{Q_i({\bf f})}(r)-\dfrac{1}{s}N_{W(F)}(r).$$
Since the number $q'$ and the dimensions $s=\dim\mathcal L(\Psi(q')), t=\dim\mathcal L(\Psi(q'+1))$ can be explicitly estimated as (\ref{4.1}) (bound above by an function of $q$ and $\epsilon$), the right hand side of the above inequality will be bounded above by a sum of the form $\sum_{i=1}^q\frac{1}{d}N^{[L(q,\epsilon)]}_{Q_i({\bf f})}(r)$, where $L(q,\epsilon)$ is explicitly estimated. Hence, we may get the second main theorem with truncation level for counting functions as follows.
\begin{theorem}
Let $f$ be a nonconstant meromorphic map of $\mathbf{C}^m$ into $\P^N(\mathbf{C})$ with a reduced representation ${\bf f}=(f_0,\ldots,f_N)$. Let $\{Q_i\}_{i=1}^q$ be a family of slowly (with respect to $f$) moving hypersurfaces in $\P^N(\C)$ with $\deg Q_i = d_i\ (1\le i\le q).$ Let $n_f$ be the algebraic dimension of $f$ over $\mathcal K_{\{Q_i\}_{i=1}^q}$ and let $\Delta_f$ be the distributive constant of $\{Q_i\}_{i=1}^q$ with respect to $f$. Suppose that $f$ has finite algebraic degree over $\mathcal K_{\{Q_i\}_{i=1}^q}$.  Then for any $\epsilon >0$, there exists a positive integer $L(q,\epsilon)$ such that
$$\|\ (q-\Delta_f(n_f+1)-\epsilon)T_f(r)\le\sum_{i=1}^q\frac{1}{d}N^{[L(q,\epsilon)]}_{Q_i({\bf f})}(r).$$
\end{theorem}

(b) If the image of $f$ is contained in a subvariety $V$ of $\P^N(\C)$, then $1\le n_f\le \dim V$ and by Remark \ref{3.5} we have $\Delta_f\le (\dim V-n_f+1)\Delta_V,$ and hence
$$\|\ (q-\Delta_V(\dim V-n_f+1)(n_f+1)-\epsilon)T_f(r)\le\sum_{i=1}^q\frac{1}{d}N^{[L(q,\epsilon)]}_{Q_i({\bf f})}(r).$$
Since $(\dim V-n_f+1)(n_f+1)\le \left(\dfrac{\dim V}{2}+1\right)^2$, we get the second main theorem
$$\biggl\|\ \left(q-\Delta_V\left(\dfrac{\dim V}{2}+1\right)^2-\epsilon\right)T_f(r)\le\sum_{i=1}^q\frac{1}{d}N^{[L(q,\epsilon)]}_{Q_i({\bf f})}(r).$$
(c)  With the help of Lemma \ref{3.2}, using the construction method for the new filtralation of G. Dethloff and T. V. Tan \cite{DT2} (see also \cite[Remark 3.1]{YY}) and repeating the routine way in the proof of second main theorem, we may get the following theorem without the condition on the algebraic degree of $f$.
\begin{theorem}
Let $f$ be a nonconstant meromorphic map of $\mathbf{C}^m$ into $\P^N(\mathbf{C})$ with a reduced representation ${\bf f}=(f_0,\ldots,f_N)$. Let $\{Q_i\}_{i=1}^q$ be a family of slowly (with respect to $f$) moving hypersurfaces in $\P^N(\C)$ with $\deg Q_i = d_i\ (1\le i\le q).$ Let $n_f$ be the algebraic dimension of $f$ over $\mathcal K_{\{Q_i\}_{i=1}^q}$ and let $\Delta_f$ be the distributive constant of $\{Q_i\}_{i=1}^q$ with respect to $f$. Then for any $\epsilon >0$, we have
$$\|\ (q-\Delta_f(n_f+1)-\epsilon)T_f(r)\le\sum_{i=1}^q\frac{1}{d}N_{Q_i({\bf f})}(r).$$
\end{theorem}
However, by this method, the truncation level $[L(q,\epsilon)]$ can not be given.}
\end{remark} 

\vskip0.2cm
\noindent
{\bf Disclosure statement:} The author states that there is no conflict of interest. 

\vskip0.2cm
 \noindent
\textbf{Acknowledgments.} The author would like to thank the referees for their helpful comments and suggestions on the first version of this paper.
This research is funded by Vietnam National Foundation for Science and Technology Development (NAFOSTED) under the grant number 101.02-2021.12.

\vskip0.2cm
{\footnotesize 
\noindent
{\sc Si Duc Quang}\\
$^1$ Department of Mathematics, Hanoi National University of Education,\\
 136-Xuan Thuy, Cau Giay, Hanoi, Vietnam.\\
$^2$ Thang Long Institute of Mathematics and Applied Sciences,\\
 Nghiem Xuan Yem, Hoang Mai, HaNoi, Vietnam.\\
\textit{E-mail}: quangsd@hnue.edu.vn}

\end{document}